\documentclass{amsart}
	  	  
\usepackage{amsfonts}
\usepackage{amssymb}
\usepackage{hyperref}
\usepackage{tikz}
\usetikzlibrary{matrix,arrows,decorations.pathmorphing}

\theoremstyle{plain}
\newtheorem{theorem}{Theorem}[section]
\newtheorem{lemma}[theorem]{Lemma}
\newtheorem{cor}[theorem]{Corollary}
\newtheorem{prop}[theorem]{Proposition}

\theoremstyle{definition}

\theoremstyle{remark}

\newcommand{\lemref}[1]{\hyperref[#1]{Lemma \ref*{#1}}}
\newcommand{\thmref}[1]{\hyperref[#1]{Theorem \ref*{#1}}}
\newcommand{\propref}[1]{\hyperref[#1]{Proposition \ref*{#1}}}
\newcommand{\corref}[1]{\hyperref[#1]{Corollary \ref*{#1}}}
\newcommand{\defref}[1]{\hyperref[#1]{Definition \ref*{#1}}}
\newcommand{\remref}[1]{\hyperref[#1]{Remark \ref*{#1}}}
\newcommand{\conjref}[1]{\hyperref[#1]{Conjecture \ref*{#1}}}

\newcommand{\Ker}{\mathrm{Ker}}

\makeatletter
\newcommand*{\defeq}{\mathrel{\rlap{%
                     \raisebox{0.27ex}{$\m@th\cdot$}}%
                     \raisebox{-0.27ex}{$\m@th\cdot$}}%
                     =}

\numberwithin{equation}{section}
\makeatother

\makeatletter
\def\@setcopyright{}
\def\serieslogo@{}
\makeatother

\begin{document}

\author{Mark Shusterman}
\address{Raymond and Beverly Sackler School of Mathematical Sciences, Tel-Aviv University, Tel-Aviv, Israel}
\email{markshus@mail.tau.ac.il} 
\title{Schreier's formula for Prosupersolvable groups}
\begin{abstract}
We classify the finitely generated prosupersolvable groups that satisfy Schreier's formula for the number of generators of open subgroups.
\end{abstract}



   


\maketitle

\section{Introduction}

The famous Nielsen-Schreier theorem asserts that a subgroup $H$ of a (nonabelian) free group $F$ is itself free, 
and if $F$ is finitely generated and $[F : H] < \infty$, then 
\begin{equation} \label{S-formulaEq}
d(H) = (d(F)-1)[F : H] + 1.
\end{equation} 
Accordingly, given an integer $n > 1$, we say that a group $G$ satisfies Schreier's formula for $n$,
if for every finite index subgroup $H$ of $G$ we have 
\begin{equation} \label{SnFormulaEq}
d(H) = (n-1)[G : H] + 1.
\end{equation}
We also say that $G$ `satisfies Schreier's formula' omitting $n$ (which is equal to $d(G)$).
Following \cite[Proposition 3.6]{AS}, one can show that free groups are the only residually finite groups that satisfy Schreier's formula. 
In light of that, it is tempting to consider the profinite analogue, which has already been studied in \cite{AS, J, JL, L, LD}.

All notions considered here for profinite groups should be understood in the topological sense. 
For instance, we say that a profinite group $\Gamma$ satisfies Schreier's formula for an integer $n > 1$,
if \eqref{SnFormulaEq} holds for every open subgroup $H$ of $\Gamma$, 
where $d(H)$ is the least cardinality of a (topological) generating set for $H$. 
By \cite[Corollary 4.4]{L}, a pronilpotent group satisfies Schreier's formula if and only if it is a free pro-$p$ group for some prime number $p$. 
Yet, it is unknown which prosolvable groups satisfy Schreier's formula. 
Results in this direction include an example \cite[Example 2.6 (D)]{LD} of a prosolvable group with finite Sylow subgroups that satisfies Schreier's formula, 
and a necessary condition \cite[Theorem 5.9]{JL} for the formula to hold.
Here we confine our attention to a subclass of finite solvable groups, which contains all finite nilpotent groups.

Recall that a finite group $G$ is called supersolvable if there exists a normal series
\begin{equation} \label{SuperSolDefEq}
\{1\} = H_0 \lhd H_1 \lhd \dots \lhd H_{n-1} \lhd H_{n} = G
\end{equation}
such that for each $0 \leq i \leq n-1$, the group $H_{i+1}/H_i$ is cyclic and $H_i \lhd G$.
In other words, all the chief factors of $G$ are cyclic.
Accordingly, a profinite group $\Gamma$ is said to be prosupersolvable, 
if every finite homomorphic image of it is supersolvable. 
Some notable works exploring prosupersolvable groups are \cite{AS, OR, RZ, Sm}.

In this work we refine \cite[Theorem 1.1]{AS} by classifying the prosupersolvable groups that satisfy Schreier's formula for an integer $n > 1$. 
For a classical result recovering the structure of a group from the numbers of generators of open subgroups,
see \cite{DL}.

\begin{theorem} \label{FirstRes}

Let $\Gamma$ be a profinite group, and let $n > 1$ be an integer. 
The following two conditions are equivalent. 

\begin{enumerate}

\item $\Gamma$ is a prosupersolvable group that satisfies Schreier's formula for $n$.

\item There exists a prime $p$ and an $n$-generated finite abelian group $A$ of exponent dividing $p-1$, such that
$\Gamma \cong T \rtimes A$, where $T$ is a free pro-$p$ group on the set $\big( \{1, \dots, n-1\} \times A \big) \cup \{u\}$, and the action of $A$ on $T$ is given by 
\begin{equation*} \label{ActDefEq}
au = u, \quad a(j,b) = (j,ab)
\end{equation*}
for all $a,b \in A$ and $1 \leq j \leq n-1$.

\end{enumerate}

\end{theorem}

Some other works that study discrete and profinite groups through the numbers of generators of finite index subgroups are \cite{AJN, AN, Gi, KN, KZ, Lack2, Lack, O, Sch, Sh0, Sh1, S1}.

Additional motivation for the study of Schreier's formula comes from the connection (see \cite{LD}) to the problem of extending the Nielsen-Schreier theorem to free profinite groups, 
which has been explored, for instance, in \cite{B, FJ, H, J, JL, JL2, M, RZ, S}.
Our contribution in this direction is the following corollary of \thmref{FirstRes}.

\begin{cor} \label{SecondRes}

Let $F$ be a nonabelian finitely generated free profinite group, and let $N \lhd_c F$ be a subgroup with $F/N$ prosupersolvable. Then $N$ is a free profinite group if and only if $F/N$ does not satisfy Schreier's formula for $d(F)$.

\end{cor}

Note that the `if' part is known to be true (in general) by \cite[Theorem 3.1]{LD}.

\section{Prosupersolvable groups}

We recall some basic facts about finite supersolvable groups, and state their generalizations to profinite groups, the (routine) proofs of which are omitted. 
For example, the following is a consequence of \cite[Theorem 10.5.4]{Hall}.

\begin{cor} \label{NilComCor}

The commutator of a prosupersolvable group is pronilpotent.

\end{cor}

The next corollary follows from \cite[Proposition 4.4]{P}.

\begin{cor} \label{DivPm1Cor}

Let $G$ be a profinite group for which there exist a prime number $p$ and a normal $p$-Sylow subgroup $P$ of $G$, 
such that $G/P$ is an abelian group of exponent dividing $p-1$. Then $G$ is prosupersolvable.

\end{cor}

Let us show that the assumption on the exponent of the quotient is necessary.

\begin{prop} \label{NdivPm1Prop}

Let $p$ be a prime, let $A$ be a finite abelian group of order prime to $p$, and suppose that $e \defeq \exp(A) \nmid p-1$. 
Then there exists a nonsupersolvable finite group $G$, that is an extension of $A$ by a $p$-group, 
such that $d(G) \leq \max\{d(A),2\}$.


\end{prop}

\begin{proof}

It follows from the structure theorem for finitely generated abelian groups that $A$ is isomorphic to a direct product of cyclic groups whose orders are powers of primes, 
and $e$ is the least common multiple of these powers. 
Since $e \nmid p - 1$, we see that one of the aforementioned prime powers does not divide $p-1$.
Hence, there exists a prime power $r = q^n \nmid p-1$, and an epimorphism $\lambda \colon A \to \mathbb{Z}/r\mathbb{Z}$. Since $r \mid |A|$, it is prime to $p$, so there exists a minimal integer $k > 1$ such that $r \mid p^k - 1$. 
As $\mathbb{F}_{p^k}^*$ is cyclic, there exists some $t \in \mathbb{F}_{p^k}^*$ of order $r$. 
Denoting by $P$ the additive group of $\mathbb{F}_{p^k}$, we see that $\mathbb{Z}/r\mathbb{Z}$ acts on it via multiplication by $t$. 
This action gives $P$ the structure of an $A$-module through $\lambda$, and we define $G \defeq P \rtimes A$.

Let us show that $P$ is a simple $A$-module. 
For that, take a nontrivial submodule $R$ of $P$, and note that $R$ is closed under multiplication by $t$, and thus also under multiplication by $\mathbb{F}_p[t]$ which is some subfield $\mathbb{F}_{p^{\ell}} \subseteq \mathbb{F}_{p^k}$.
As $t \in \mathbb{F}_{p^\ell}^*$, we conclude that $r \mid p^{\ell}-1$, so the minimality of $k$ implies that $k \leq \ell$ and thus $\mathbb{F}_p[t] = \mathbb{F}_{p^k}$. 
It follows that $0 \lneq R \leq \mathbb{F}_{p^k}$ is closed under multiplication by $\mathbb{F}_{p^k}$, so $R = \mathbb{F}_{p^k}$. 

Suppose toward a contradiction that $G$ is supersolvable. 
As $P \lhd G$, there exists a chief series (with cyclic chief factors) that contains $P$. 
Since $P$ is not cyclic (as $k > 1$) our series contains a nontrivial proper subgroup of $P$ that is normal in $G$. 
This is a contradiction to the $A$-simplicity of $P$.
The bound on the number of generators of $G$ follows from \cite[Theorem C]{AG}.

\end{proof} 

\section{Schreier's formula}

We present a lemma that plays a key role in our proof of \thmref{FirstRes}. 
This lemma generalizes \cite[Main Theorem, Theorem 2.3]{J}.
A similar argument can be found in \cite[Theorem 2.8]{S1}. 

\begin{lemma} \label{KeyLem}

Let $\Gamma$ be a finitely generated nonprocyclic profinite group satisfying Schreier's formula, 
and let $H \leq_c \Gamma$ be a finitely generated subgroup containing a nontrivial normal subgroup $N \lhd_c \Gamma$. Then $[\Gamma : H] < \infty$.

\end{lemma}

\begin{proof}

Toward a contradiction, suppose that $[\Gamma : H] = \infty$, 
so the index of the image of $H$ under epimorphisms from $\Gamma$ to finite groups can be arbitrarily large. 
For instance, there exists some $V \lhd_o \Gamma$ such that
\begin{equation} \label{VIndexEq}
[\Gamma : HV] > \frac{2d(H)}{d(\Gamma)-1}.
\end{equation} 
Since $N \neq 1$, we can take a $U \lhd_o \Gamma$ that is contained in $V$ but does not contain $N$.
Using \cite[Lemma 2.7]{S1}, and \cite[Corollary 3.6.3]{RZ} we obtain
\begin{equation}
\begin{split}
d(U) &\leq d(H \cap U) + d(U/N \cap U) \leq d(H)[H : H \cap U]  + d(UN/N) \\
&\leq d(H)[HU : U] + d(UN) = \frac{d(H)[\Gamma : U]}{[\Gamma : HU]} + (d(\Gamma)-1)[\Gamma : UN] + 1 \\ 
&\leq \frac{d(H)[\Gamma : U]}{[\Gamma : HV]} + \frac{(d(\Gamma)-1)[\Gamma : U]}{2} + 1 \stackrel{\ref{VIndexEq}}{<} (d(\Gamma)-1)[\Gamma : U] + 1
\end{split}
\end{equation}
so we have a contradiction to the fact that Schreier's formula holds for $\Gamma$.
\end{proof} 

\section{The proof of \thmref{FirstRes} and its corollary}

Let us start by proving the easier direction, namely $(2) \Rightarrow (1)$.

\begin{proof}

Clearly, $T$ is a pro-$p$ normal subgroup of $\Gamma$ and $\Gamma/T \cong A$ that is of order prime to $p$, 
as its exponent divides $p-1$.
Thus, by \corref{DivPm1Cor}, $\Gamma$ is prosupersolvable.
Since $T \lhd_c \Gamma$, 
we infer from \cite[7.3.17]{Sc} that the Frattini subgroups satisfy $\varPhi(T) \leq \varPhi(\Gamma)$.
We can thus bound the number of generators of $\Gamma$ from above by 
\begin{equation} \label{FrattiniGenEq}
\begin{split}
d(\Gamma) &= d(\Gamma/\varPhi(\Gamma)) \leq d(\Gamma/\varPhi(T)) = d(T/\varPhi(T) \rtimes A) \\
&\stackrel{\ref{ActDefEq}}{=} d(C_p \times (C_p^{n-1} \ \mathrm{wr} \ A)) \leq n
\end{split}
\end{equation}
where $C_p$ is the cyclic group of order $p$, and the last inequality is a consequence of \cite[Satz 4]{G}.
Since $T$ is free pro-$p$, it satisfies Schreier's formula for $(n-1)|A| + 1$, so by \cite[Lemma 2.4]{LD},
$\Gamma$ satisfies Schreier's formula for $n$.

\end{proof}

Now we prove that $(1) \Rightarrow (2)$. 

\begin{proof}

Let $\Gamma'$ be the commutator of $\Gamma$ and note that $\Gamma' \neq \{1\}$ since otherwise $\Gamma$ would have been abelian, and would not satisfy Schreier's formula, as can be seen from \cite[Proposition 4.3.6]{RZ}.
Therefore, there exists a prime number $p$ and a nontrivial $p$-Sylow subgroup $N$ of $\Gamma'$.
By \corref{NilComCor}, $N$ is a characteristic subgroup of $\Gamma'$ and thus a normal subgroup of $\Gamma$ (since $\Gamma' \lhd_c \Gamma$).
Let $T$ be a $p$-Sylow subgroup of $\Gamma$ containing $N$, and note that it is finitely generated by \cite[Corollary 3.9]{OR}.
It follows from \lemref{KeyLem} that $[\Gamma : T] < \infty$, so $\Gamma$ is virtually pro-$p$.

In the preceding paragraph we have in fact shown that for every prime number $q \mid |\Gamma'|$, the group $\Gamma$ is virtually pro-$q$.
Since $\Gamma$ satisfies Schreier's formula, it is infinite, so it is virtually pro-$q$ for at most one prime $q$ (namely, $p$).
In other words, $\Gamma' = N$ so $A \defeq \Gamma/T$ is a (finite) abelian group since $T$ contains $N$. 
Furthermore, by \cite[Lemma 2.4]{LD}, $T$ satisfies Schreier's formula for $(n-1)|A| + 1$, 
so we can think of $T$ as the free pro-$p$ group on the set $\big( \{1, \dots, n-1\} \times A \big) \cup \{u\}$. 

Let $F$ be a free profinite group on $n$ generators, 
let $\gamma \colon F \to \Gamma$ be a surjection, 
and set $M \defeq \Ker(\gamma), \ U \defeq \gamma^{-1}(T)$.
Let $R_p(U)$ be the intersection of all open subgroups of $U$ whose index is a power of $p$. By \cite[Lemma 3.4.1]{RZ}, $U/R_p(U)$ is the largest pro-$p$ quotient of $U$, 
so in particular $R_p(U) \leq M$ as $U/M \cong T$ is a pro-$p$ group.
By \cite[Theorem 3.6.2]{RZ}, $U$ is a free profinite group on $(n-1)|A| + 1$ generators, 
so by \cite[Proposition 3.4.2]{RZ}, $U/R_p(U)$ is a free pro-$p$ group on the same number of generators.
Since $R_p(U) \leq M$, the group $U/R_p(U)$ surjects onto $U/M \cong T$, 
which is a free pro-$p$ group with the same number of generators as $U/R_p(U)$ (and is thus isomorphic to $U/R_p(U)$).
By \cite[Proposition 2.5.2]{RZ}, this surjection is injective, so its kernel $M/R_p(U)$ is trivial, 
and we conclude that $M = R_p(U)$.

Suppose toward a contradiction that $\exp(A) \nmid p-1$.
By \propref{NdivPm1Prop}, there exists an $n$-generated nonsupersolvable finite group $G$, 
that is an extension of $A$ by a $p$-group.
It follows from \cite[Theorem 3.5.8]{RZ}, that the embedding problem
\begin{equation*}
\begin{tikzpicture}[scale=1.5]
\node (B) at (1.6,1.4) {$F$};
\node (C) at (0,0) {$G$};
\node (D) at (1.6,0) {$A$};
\path[->,font=\scriptsize,>=angle 90]
(B) edge node[right]{$\mod U$} (D)
(C) edge node[above]{} (D);
\end{tikzpicture}
\end{equation*} 
is properly solvable, by an epimorphism $\tau \colon F \to G$.
Since $\tau$ is a solution, we see that $U/\Ker(\tau)$ is a $p$-group, so $\Ker(\tau) \geq R_p(U) = M$.
We conclude that $\Gamma \cong F/M$ surjects onto $G \cong F/\Ker(\tau)$, 
which contradicts the prosupersolvability of $\Gamma$.

Set $\Lambda \defeq T \rtimes A$, where the action of $A$ on $T$ is the one described in \thmref{FirstRes}.
Since we already know that $(2) \Rightarrow (1)$ in \thmref{FirstRes}, we infer that $\Lambda$ is $n$-generated,
so the embedding problem
\begin{equation*}
\begin{tikzpicture}[scale=1.5]
\node (B) at (1.6,1.4) {$F$};
\node (C) at (0,0) {$\Lambda$};
\node (D) at (1.6,0) {$A$};
\path[->,font=\scriptsize,>=angle 90]
(B) edge node[right]{$\mod U$} (D)
(C) edge node[above]{$\mod T$} (D);
\end{tikzpicture}
\end{equation*} 
is properly solvable by an epimorphism $\lambda \colon F \to \Lambda$, in view of \cite[Theorem 3.5.8]{RZ}.
Since $U/\Ker(\lambda) \cong T$, 
using the same reasoning as for $\Ker(\gamma)$, we conclude that $\Ker(\lambda) = R_p(U)$ as well.
Hence, 
\begin{equation}
\Gamma \cong F/\Ker(\gamma) = F/R_p(U) = F/\Ker(\lambda) \cong \Lambda.
\end{equation}
\end{proof}

Let us now deduce \corref{SecondRes}.

\begin{proof}

Suppose that $F/N$ satisfies Schreier's formula for $d(F)$.
By \thmref{FirstRes}, 
there exists a subgroup $N \leq U \lhd_o F$ such that $U/N$ is a free pro-$p$ group on a set of cardinality $(d(F) - 1)[F : U] + 1$. 
An argument from the proof of \thmref{FirstRes} shows that $N = R_p(U)$, so by \cite[Lemma 3.4.1 (e)]{RZ}, 
$N$ does not surject onto $\mathbb{Z}/p\mathbb{Z}$ and it is thus not a free profinite group.

\end{proof}

\section{Acknowledgments}

The author would like to thank Lior Bary-Soroker for many helpful conversations.

\end{document}